\documentclass[11pt,reqno]{amsart}

\numberwithin{equation}{section}

\usepackage{color,graphics,epsfig}

\newcommand{\calA}{\mathcal{A}}

\newcommand{\calL}{\mathcal{L}}

\newcommand{\calW}{\mathcal{W}}

\newcommand{\mC}{\mathbb{C}}
\newcommand{\mD}{\mathbb{D}}

\newcommand{\mF}{\mathbb{F}}

\newcommand{\mR}{\mathbb{R}}
\newcommand{\mS}{\mathbb{S}}
\newcommand{\mT}{\mathbb{T}}

\newcommand{\mZ}{\mathbb{Z}}

\newcommand{\inv}{{\textrm{inv }}}
\newcommand{\nm}{\,\rule[-.6ex]{.13em}{2.3ex}\,}

\newtheorem{theorem}{Theorem}[section]
\newtheorem{lemma}[theorem]{Lemma}
\newtheorem{corollary}[theorem]{Corollary}
\newtheorem{proposition}[theorem]{Proposition}

\theoremstyle{definition}

\newtheorem{remark}[theorem]{Remark}

\theoremstyle{definition}
\newtheorem{definition}[theorem]{Definition}

\theoremstyle{definition}

\begin{document}

\keywords{}

\subjclass{Primary 93B36; Secondary 93D15, 46J15}

\title[Extension of the $\nu$-metric]{Extension of the $\nu$-metric}

\author{Joseph A. Ball}
\address{Department of Mathematics,
Virginia Tech.,
Blacksburg,  VA 24061,
USA.}
\email{joball@math.vt.edu}

\author{Amol J. Sasane}
\address{Department of Mathematics, Royal Institute of Technology,
    Stockholm, Sweden.}
\email{sasane@math.kth.se}

\begin{abstract}
  We extend the $\nu$-metric introduced by Vinnicombe
  in robust control theory for rational plants to the case of
  infinite-dimensional systems/classes of nonrational transfer
  functions. 
\end{abstract}

\maketitle

\section{Introduction}

\noindent The general {\em stabilization problem} in control theory is as
follows. Suppose that $R$ is a commutative integral domain with identity
(thought of as the class of stable transfer functions) and let
$\mF(R)$ denote the field of fractions of $R$. The stabilization
problem is:

\medskip

\begin{center}
\parbox[r]{11cm}{Given $P\in (\mF(R))^{p\times
  m}$ (an unstable plant transfer function),

 find $C \in (\mF(R))^{m\times p}$ (a stabilizing controller
  transfer function),

 such that (the closed loop transfer function)
$$
H(P,C):= \left[\begin{array}{cc} P \\ I \end{array} \right]
(I-CP)^{-1} \left[\begin{array}{cc} -C & I \end{array} \right]
$$
belongs to $R^{(p+m)\times (p+m)}$ (is stable).}
\end{center}

\medskip

\noindent Recipes for constructing such $C$ is a central theme in
control theory; see for example the book by Vidyasagar \cite{Vid}.

However, in the {\em robust stabilization problem}, one goes a step
further. One knows that the plant is just an approximation of reality,
and so one would really like the controller $C$ to not only stabilize
the {\em nominal} plant $P_0$, but also all sufficiently close plants
$P$ to $P_0$.  The question of what one means by ``closeness'' of
plants thus arises naturally. So one needs a function $d$ defined on
pairs of stabilizable plants such that
\begin{enumerate}
\item $d$ is a metric on the set of all stabilizable
plants,
\item $d$ is amenable to computation, and
\item $d$ has ``good''properties in the robust stabilization problem.
\end{enumerate}
Such a desirable metric, was introduced by Glenn Vinnicombe in
\cite{Vin} and is called the $\nu$-{\em metric}. In that paper,
essentially $R$ was taken to be the rational functions without poles in the
closed unit disk  or, more generally,  the disk algebra, and the most important results
were that the $\nu$-metric is indeed a metric on the set of
stabilizable plants, and moreover, it has the following nice property
in the context of the robust stabilization problem:

\medskip

\parbox[r]{12cm}{(P): If the $\nu$-metric between two stabilizable
  plants $P_0$ and $P$ is less than the stability margin $\mu_{P_0,C}$
  of $P_0$ and its stabilizing controller $C$, then $C$ also stabilizes $P$.}

\medskip

\noindent The problem of what happens when $R$ is some other ring of
stable transfer functions of infinite-dimensional (that is, one time
axis and infinite-dimensional state space) or multidimensional
systems (several ``time'' axes of evolution) was left open. This problem of extending the $\nu$-metric from
the rational case to transfer function classes of infinite-dimensional
systems was also mentioned in article by Nicholas Young \cite{You}. In
this article, we address this issue of extending the $\nu$-metric.

The starting point for our approach is abstract:  we suppose that
 $R$ is any commutative  integral domain with identity which is a subset of a Banach algebra
$S$ satisfying certain assumptions, which we label (A1)-(A4). We then
define an ``abstract'' $\nu$-metric in this setup, and show that it does
define a metric on the class of all stabilizable plants. We also show
that it has the desired property (P) in the context of robust
stabilization for an appropriate definition of stability margin $\mu_{P_0,C}$.

Next we give several examples of integral domains $R$ arising as natural
classes of stable transfer functions of infinite-dimensional and
multidimensional systems which satisfy the abstract assumptions (A1)
to (A4). In particular, we cover the case of full subalgebras of the
disk algebra, the causal almost periodic function classes, the class
of measures on $[0 ,+\infty)$ without a singular nonatomic part, and
the polydisk algebra.

The paper is organized as follows:
\begin{enumerate}
\item In Section \ref{section_set_up}, we give our general setup and
  assumptions, and define the abstract metric $d_\nu$.
\item In Section \ref{section_metric}, we will show that $d_\nu$ is a
  metric on the set of stabilizable plants.
\item In Section \ref{section_robustness}, we introduce a notion of stability margin $\mu_{P,C}$
and  prove Theorem~\ref{theorem_rst};  this implies in particular that if the
  $\nu$-metric between two stabilizable plants $P_0$ and $P$ is less
  than the stability margin $\mu_{P_0,C}$ of $P$ and its stabilizing controller
  $C$, then $C$ also stabilizes $P$.
\item In Section \ref{section_applications}, we specialize
  $R$ to concrete rings of stable transfer functions of various types,
  and show that our abstract assumptions hold in these particular
  cases.
  \item The final Section \ref{section_further} mentions a loose end which is a direction for further work.
\end{enumerate}

\section{General setup and assumptions}
\label{section_set_up}

\noindent Our setup is the following:
\begin{itemize}
\item[(A1)] $R$ is commutative integral domain with identity.
\item[(A2)] $S$ is a unital commutative complex semisimple Banach
  algebra with an involution $\cdot^*$, such that $R \subset S$.  We
  use $\inv S$ to denote the invertible elements of $S$.
\item[(A3)] There exists a map $\iota: \inv S \rightarrow G$, where
  $(G,+)$ is an Abelian group with identity denoted by $\circ$, and
  $\iota$ satisfies
\begin{itemize}
\item[(I1)] $\iota(ab)= \iota (a) +\iota(b)$ ($a,b \in \inv S$).
\item[(I2)] $\iota(a^*)=-\iota(a)$ ($a\in \inv S$).
\item[(I3)] $\iota$ is locally constant, that is, $\iota$ continuous
  when $G$ is equipped with the discrete topology.
\end{itemize}
\item[(A4)] $x\in R \cap (\inv S)$ is invertible as an element of $R$
  iff $\iota(x)=\circ$.
\end{itemize}

\noindent A consequence of (I3) is the following ``homotopic
invariance of the index'', which we will use in the sequel.

\begin{proposition}
\label{prop_hom_inv}
If $H:[0,1] \rightarrow \textrm{\em inv } S$ is a continuous map, then
$$
\iota (H(0))=\iota(H(1)).
$$
\end{proposition}
\begin{proof}
  The map $h$, given by $t \mapsto \iota(H(t)): [0,1] \rightarrow G$
  is continuous. Here $[0,1]$ is equipped with usual topology from
  $\mR$, while $G$ is equipped with the discrete topology, given by
  the metric
$$
 d(x,y)=
 \left\{ \begin{array}{ll}
 1 & \textrm{if } x\neq y,\\
 0 & \textrm{if } x= y,
 \end{array} \right.  \quad (x,y \in G).
 $$
 The image of the connected set $[0,1]$ under the continuous map $h$
 is connected. But the only connected subsets of $G$ are the singleton
 sets, since $G$ is carrying the discrete topology. Hence
 $\iota(H(0))=\iota(H(1)$.
\end{proof}

We recall the following standard definitions from the factorization
approach to control theory.

\begin{definition}
$\;$

\noindent {\bf The notation $\mF(R)$:} $\mF(R)$ denotes the field of
  fractions of $R$.

\medskip

\noindent {\bf The notation $F^*$:} If $F\in R^{p\times m}$, then $F^*\in
  S^{m\times p}$ is the matrix with the entry in the $i$th row and
  $j$th column given by $F_{ji}^*$, for all $1\leq i\leq p$, and all $
  1\leq j \leq m$.

\medskip

\noindent {\bf Right coprime/normalized coprime factorization:} Given
a matrix $P \in (\mF(R))^{p\times m}$, a factorization $P=ND^{-1}$,
where $N,D$ are matrices with entries from $R$, is called a {\em right
  coprime factorization of} $P$ if there exist matrices $X, Y$ with
entries from $R$ such that $ X N + Y D=I_m$.  If moreover there holds
that $ N^{*} N +D^{*} D =I_m$, then the right coprime factorization is
referred to as a {\em normalized} right coprime factorization of $P$.

\medskip

\noindent {\bf Left coprime/normalized coprime factorization:}
Similarly, a factorization $P=\widetilde{D}^{-1}\widetilde{N}$, where
$\widetilde{N},\widetilde{D}$ are matrices with entries from $R$, is
called a {\em left coprime factorization of} $P$ if there exist
matrices $\widetilde{X}, \widetilde{Y}$ with entries from $R$ such
that $ \widetilde{N} \widetilde{X}+\widetilde{D} \widetilde{Y}=I_p.  $
If moreover there holds that $ \widetilde{N} \widetilde{N}^{*}
+\widetilde{D}\widetilde{D}^{*}=I_p, $ then the left coprime
factorization is referred to as a {\em normalized} left coprime
factorization of $P$.  We note that the existence of both a left and
right normalized factorization $P = N D^{-1} = \widetilde{D}^{-1}
\widetilde{N}$ for $P$ leads immediately to a normalized {\em double
  coprime factorization} of $P$, i.e., one has the identity
  \begin{equation}  \label{doublecoprime}
  \begin{bmatrix} N^*    & D^* \\ -\widetilde{D} & \widetilde{N} \end{bmatrix}
  \begin{bmatrix} N & -\widetilde{D}^* \\ D & \widetilde{N}^* \end{bmatrix}
  = \begin{bmatrix} I & 0 \\ 0 & I \end{bmatrix}.
  \end{equation}
  Since we are dealing with finite matrices over a commutative ring,
  \eqref{doublecoprime} implies also the identity
  \begin{equation}  \label{doublecoprime'}
  \begin{bmatrix} N & -\widetilde{D}^* \\ D & \widetilde{N}^* \end{bmatrix}
  \begin{bmatrix} N^*    & D^* \\ -\widetilde{D} & \widetilde{N} \end{bmatrix} =
   \begin{bmatrix} I & 0 \\ 0 & I \end{bmatrix}.
   \end{equation}

\medskip

\noindent {\bf The notation $G, \widetilde{G}, K,\widetilde{K}$:}
Given $P \in (\mF(R))^{p\times m}$ with normalized right and left
factorizations $P=N D^{-1}$ and $P= \widetilde{D}^{-1} \widetilde{N}$,
respectively, we introduce the following matrices with entries from
$R$:
  $$
  G=\left[ \begin{array}{cc} N \\ D \end{array} \right] \quad
  \textrm{and} \quad
  \widetilde{G}=\left[ \begin{array}{cc} -\widetilde{D} &
  \widetilde{N} \end{array} \right] .
  $$
  In this notation the fact that the left and right coprime
  factorizations of $P$ are normalized translates to
  \begin{equation} \label{norcp}
  G^*G = I, \quad \widetilde{G} \widetilde{G}^* = I
  \end{equation}
  and the identity \eqref{doublecoprime'} assumes the form
  \begin{equation}  \label{dcp}
  G G^* + \widetilde{G}^* \widetilde{G} = I.
  \end{equation}
  Similarly, given $C \in (\mF(R))^{m\times p}$ with normalized right
  and left factorizations $C=N_C D_C^{-1}$ and $C=
  \widetilde{D}_C^{-1} \widetilde{N}_C$, respectively, we introduce
  the following matrices with entries from $R$:
  $$
  K=\left[ \begin{array}{cc} D_C \\ N_C \end{array} \right] \quad
  \textrm{and} \quad
  \widetilde{K}=\left[ \begin{array}{cc}
      -\widetilde{N}_C & \widetilde{D}_C \end{array} \right] .
  $$

\medskip

\noindent {\bf The notation $\mS(R,p, m)$:} We denote by $\mS(R,p, m)$ the
  set of all elements  $P\in (\mF(R))^{p\times m}$ that posses a normalized
  right coprime factorization and a normalized left coprime
  factorization.

\end{definition}

\begin{remark}
Given $P \in (\mF(R))^{p\times m}$ and $C\in (\mF(R))^{m\times p}$,
define the {\em closed loop transfer function}
$$
H(P,C):= \left[\begin{array}{cc} P \\ I \end{array} \right]
(I-CP)^{-1} \left[\begin{array}{cc} -C & I \end{array} \right] \in
(\mF(R))^{(p+m)\times (p+m)}.
$$
It can be shown (see for example \cite[Chapter 8]{Vid}) that if $ P\in
\mS(R,p,m)$, then $P$ is a {\em stabilizable plant}, that is,
\begin{equation}   \label{dc/stab}
\mS(R,p,m)
 \subset \left\{ P\in (\mF(R))^{p\times m} \bigg|
\begin{array}{ll}
   \exists C\in (\mF(R))^{m\times p} \textrm{ such that} \\
   H(P,C)\in R^{(p+m)\times (p+m)}
\end{array}
\right\}.
\end{equation}
It was shown by A. Quadrat \cite[Theorem 6.3]{Qua} that if the Banach
algebra $R$ is a projective-free ring, then every stabilizable plant
admits a right coprime factorization and a left coprime factorization,
that is, the reverse containment $\supset$ and hence equality holds in
\eqref{dc/stab}.
\end{remark}

We will need a couple of straightforward results on coprime
factorizations, which we have listed below.  The first lemma says that
coprime factorizations are unique up to invertibles.

\goodbreak

\begin{lemma}
Let $P\in (\mF(R))^{p\times m}$.
\begin{enumerate}
\item If $P$ has right coprime factorizations $ P= N_1 D_1^{-1}=N_2
  D_2^{-1}$, then there exist $V, \Lambda \in R^{m\times m}$ such that
  $ V\Lambda =\Lambda V =I_m$, $N_1 =N_2 V$ and $D_1= D_2 V$.
\item If $P$ has left coprime factorizations $ P= \widetilde{D}_1^{-1}
  \widetilde{N}_1=\widetilde{D}_2^{-1}\widetilde{N}_2$, then there
  exist $\widetilde{V}, \widetilde{\Lambda} \in R^{p\times p}$ such
  that $ \widetilde{V}\widetilde{\Lambda} =\widetilde{\Lambda}
  \widetilde{V} =I_p$, $\widetilde{N}_1 =\widetilde{V}\widetilde{N}_2$
  and $\widetilde{D}_1= \widetilde{V}\widetilde{D}_2 $.
\end{enumerate}
\end{lemma}

In the case of normalized coprime factorizations, the invertibles can
be chosen to be unitary.

\begin{lemma}
\label{lemma_ncfr}
Let $P\in (\mF(R))^{p\times m}$.
\begin{enumerate}
\item If $P$ has normalized right coprime factorizations $P\!=\! N_1
  D_1^{-1}\!\!=\!N_2 D_2^{-1}\!\!$, then there exists a $U \in
  R^{m\times m}$, which is invertible as an element of $R^{m\times
    m}$, and such that $ U^* U =UU^* =I_m$, $N_1 =N_2 U$ and $D_1= D_2
  U$.
\item If $P$ has normalized left coprime factorizations $ P=
  \widetilde{D}_1^{-1}\widetilde{N}_1=\widetilde{D}_2^{-1}\widetilde{N}_2$,
  then there exists a $\widetilde{U} \in R^{p\times p}$ which is
  invertible as an element of $R^{p\times p}$, and such that $
  \widetilde{U}^{*}\widetilde{U} =\widetilde{U}\widetilde{U}^* =I_p$,
  $\widetilde{N}_1 =\widetilde{U}\widetilde{N}_2$ and
  $\widetilde{D}_1= \widetilde{U}\widetilde{D}_2 $.
\end{enumerate}
\end{lemma}

\begin{lemma}
\label{lemma_fth}
Suppose that $F\in R^{m\times m}$, $\det F \in \textrm{\em inv }S$ and
$\iota( \det F)=\circ$. Then $F$ is invertible as an element of $
R^{m\times m}$.
\end{lemma}
\begin{proof} Since $\det F \in \inv S$ and $\iota( \det F)=\circ$, it
  follows from (A4) that $\det F $ is invertible as an element of $R$.
  The result then follows from Cramer's rule.
\end{proof}

We now define the metric $d_\nu$ on $\mS(R, p, m)$. But first we
specify the norm we use for matrices with entries from $S$.

\begin{definition}[$\|\cdot \|_{\infty}$]\label{def_sup_norm}
  Let $\mathfrak{M}$ denote the maximal ideal space of the Banach
  algebra $S$.  For a matrix $M \in S^{p\times m}$, we set
\begin{equation}
\label{norm}
\|M\|_{\infty}= \max_{\varphi \in \mathfrak{M}} \nm {\mathbf M}(\varphi) \nm.
\end{equation}
Here ${\mathbf M}$ denotes the entry-wise Gelfand transform of $M$,
and $\nm \cdot \nm$ denotes the induced operator norm from $\mC^{m}$
to $\mC^{p}$. For the sake of concreteness, we fix the standard
Euclidean norms on the vector spaces $\mC^{m}$ to $\mC^{p}$.
\end{definition}

The maximum in \eqref{norm} exists since $\mathfrak{M}$ is a compact
space when it is equipped with Gelfand topology, that is, the
weak-$\ast$ topology induced from $\calL( S; \mC)$. Since we have
assumed $S$ to be semisimple, the Gelfand transform
$$
\widehat{\;\cdot\;}:S \rightarrow\widehat{S}\; ( \subset
C(\mathfrak{M},\mC))
$$
is an isomorphism. If $M\in S^{1\times 1}=S$, then we note that there
are two norms available for $M$: the one as we have defined above,
namely $\|M\|_{\infty}$, and the norm $\|\cdot\|$ of $M$ as an element
of the Banach algebra $S$. But throughout this article, we will use
the norm given by \eqref{norm}.

\begin{definition}[Abstract $\nu$-metric $d_\nu$]\label{def_nu_metric}
  For $P_1, P_2 \in \mS(R,p,m)$, with the normalized left/right
  coprime factorizations
\begin{eqnarray*}
P_1&=& N_{1} D_{1}^{-1}= \widetilde{D}_{1}^{-1} \widetilde{N}_{1},\\
P_2&=& N_{2} D_{2}^{-1}= \widetilde{D}_{2}^{-1} \widetilde{N}_{2},
\end{eqnarray*}
we define
\begin{equation}
\label{eq_nu_metric}
d_{\nu} (P_1,P_2 ):=\left\{
\begin{array}{ll}
  \|\widetilde{G}_{2} G_{1}\|_{\infty} &
  \textrm{if } \det(G_1^* G_2) \in \inv S \textrm{ and }
  \iota (\det (G_1^* G_2))=\circ, \\
  1 & \textrm{otherwise}. \end{array}
\right.
\end{equation}
\end{definition}

Normalized coprime factorizations are not unique for a given plant in $\mS(R,p,m)$.  But we have the following:

\begin{lemma}
\label{lemma_wd}
$d_\nu$ given by \eqref{eq_nu_metric} is well-defined.
\end{lemma}
\begin{proof} This follows from Lemma \ref{lemma_ncfr}.
\end{proof}

\begin{lemma}
\label{lemma_bd_1}
$d_\nu$ given by \eqref{eq_nu_metric} is bounded above by $1$.
\end{lemma}
\begin{proof} We have $\|\widetilde{G}_2 G_1\|_{\infty}\leq
  \|\widetilde{G}_2 \|_{\infty}\|G_1\|_{\infty}$. As $G_1^*
  G_1=I_m$ and $\widetilde{G}_2 \widetilde{G}_2^*=I_p$, we see that
  $\|G_1\|_{\infty}=1=\|\widetilde{G}_2 \|_{\infty}$.
\end{proof}

In Section \ref{section_metric}, we will also prove the following.

\begin{theorem}
\label{thm_d_nu_is_a_metric}
$d_\nu$ given by \eqref{eq_nu_metric} is a metric on $\mS(R, p, m)$.
\end{theorem}

We recall the definition of singular values of a square matrix, and a
few properties which will be needed in the sequel.

\begin{definition}
  If $M \in \mC^{k\times k}$, then the set of eigenvalues of $MM^*$
  and $M^* M$ are equal and the eigenvalues are real. The square roots
  of these eigenvalues are called the {\em singular values of} $M$,
  and the largest of these is denoted by $\overline{\sigma}(M)$, while
  the smallest of these is denoted by $\underline{\sigma}(M)$.
\end{definition}

\begin{proposition}
\label{prop_sv}
The following hold for $P, Q \in \mC^{k \times k}$.
\begin{itemize}
\item[(S1)] $\|P\|= \overline{\sigma}(P)$.
\item[(S2)] If $P$ is invertible, then $\underline{\sigma}(P) >0$, and
  $\|P^{-1}\|= (\underline{\sigma}(P))^{-1}$.
\item[(S3)] $| \underline{\sigma}(P+Q) -\underline{\sigma}(P) | \leq
  \overline{\sigma}(Q)$.
\item[(S4)] $\overline{\sigma}(PQ) \leq \overline{\sigma}(P) \cdot
  \overline{\sigma}(Q)$.
\item[(S5)] $\underline{\sigma}(PQ) \geq \underline{\sigma}(P) \cdot
  \underline{\sigma}(Q)$.
\item[(S6)]
  $\overline{\sigma}(PQ)=\overline{\sigma}((P^*P)^\frac{1}{2} Q)=
  \overline{\sigma}(P(QQ^*)^\frac{1}{2})$.
\end{itemize}
\end{proposition}
\begin{proof} (S1), (S2) follow from the spectral theorem. (S3), (S4),
  (S5) are given in \cite[Proposition 9.6.8, Corollary
  9.6.6]{Ber}. (S6) can be verified directly using the definition of
  $\overline{\sigma}$.
\end{proof}

\begin{lemma}
\label{lemma_0.12}
Suppose that $A, B \in \mC^{p\times m}$ and that $A^* A +B^* B=I$. Then
$ (\underline{\sigma}(A))^2+ (\overline{\sigma}(B))^2 =1$.
\end{lemma}
\begin{proof}
  This follows from the spectral theorem. Indeed,  $A^* A=I-B^*
  B$, and so for all $x\in \mC^m$ with unit norm, we have
  $$
  \langle A^* A x , x \rangle = \langle x, x \rangle - \langle B^* B
  x, x\rangle=1-\langle B^* B x, x\rangle.
  $$
  Thus $ (\overline{\sigma}(A))^2=\displaystyle
  \max_{\scriptscriptstyle\|x\|=1}\langle A^* A x , x
  \rangle=1-\min_{\scriptscriptstyle \|x\|=1}\langle B^* B x, x\rangle
  =1- (\underline{\sigma}(B))^2$.
\end{proof}

In particular, we have the following consequence as an application of
this lemma.  (In this article, we often suppress the argument of the
Gelfand transforms of matrices with $S$-entries.)

\begin{lemma}
\label{lemma_0.13}
If $P_1, P_2 \in \mS(R, p, m)$, then $ \big(\underline{\sigma}(
\mathbf{G_2^* G_1} )\big)^2 + \big(\overline{\sigma}(
\mathbf{\widetilde{G}_2 G_1} )\big)^2 = 1 $ pointwise on
$\mathfrak{M}$.
\end{lemma}
\begin{proof} Observing from \eqref{norcp} and \eqref{dcp}
that $G_1^* G_1=I$ and $G_2 G_2^*
  +\widetilde{G}_2^* \widetilde{G_2}=I$, we obtain
$$
\mathbf{G_1^*G_2 G_2^* G_1}+ \mathbf{G_1^* \widetilde{G}_2^*
  \widetilde{G}_2 G_1} =I
$$
pointwise on $\mathfrak{M}$. An application of Lemma \ref{lemma_0.12}
now yields the result.
\end{proof}

\section{$d_\nu$ is a metric}
\label{section_metric}

\noindent In this section, we will prove Theorem \ref{thm_d_nu_is_a_metric}.

\medskip

\noindent{\em Proof} (of Theorem~\ref{thm_d_nu_is_a_metric}).

\subsection{Positivity}

If $P_1, P_2 \in \mS(R, p, m)$, then clearly $d_{\nu}(P_1,P_2)\geq 0$.
Also, if $d_{\nu}(P_1, P_2)=0$, then $\|\widetilde{G}_{2}
G_{1}\|_{\infty} =0$, and so $\widetilde{G}_{2} G_{1}=0$. But
$$
\widetilde{G}_{2} G_{1}= \widetilde{D}_2 (P_2 -P_1) D_1.
$$
Thus $P_1=P_2$. Finally, for $P \in \mS(R, p, m)$, it is clear that
$d_\nu(P,P)=0$.

\subsection{Symmetry}

Let $P_1, P_2 \in \mS(R, p, m)$. Since $G_1^* G_2 = (G_2^* G_1)^*$, it
follows that $\det (G_1^* G_2) $ is invertible as an element of $S$
iff $\det (G_2^* G_1)$ is invertible as an element of $S$. Using (I2),
we see that $ \iota (\det (G_1^* G_2) )=\circ$ iff $\iota (\det (G_2^*
G_1) )=\circ$.  Hence $d_\nu(P_1, P_2)=d_\nu(P_2, P_1)$.

\subsection{The triangle inequality}

Suppose that $P_1, P_2, P_0 \in \mS(R, p, m)$. We want to show that $
d_\nu (P_1, P_2) \leq d_\nu (P_1, P_0) +d_\nu (P_0, P_2)$.  Since
$d_\nu$ is bounded above by $1$, this inequality is trivially
satisfied if either $d_\nu(P_1, P_0)=1$ or $d_\nu (P_0, P_2)=1$. So in
the rest of this subsection, we will assume that $d_\nu(P_1, P_0)<1$
and $d_\nu (P_0, P_2)<1$. This means that
\begin{enumerate}
\item $\det (G_{1}^* G_{0})$ is invertible in $S$ and $\iota (\det
  (G_{1}^* G_{0}))=\circ$.
\item $\det (G_{0}^* G_{2})$ is invertible in $S$ and $\iota (\det
  (G_{0}^* G_{2}))=\circ$.
\end{enumerate}
We will consider separately the following two possible cases:

\medskip

\noindent $\underline{1}^\circ$ $\det (G_1^* G_2) \in \inv S$ and
$\iota(\det (G_1^* G_2))=\circ$.  Then $d_\nu(P_1,P_2)=
\|\widetilde{G}_2 G_1\|_{\infty}$.

\medskip

\noindent $\underline{2}^\circ$ $\neg[\det (G_1^* G_2) \in \inv S
\textrm{ and }\iota(\det (G_1^* G_2))=\circ]$. Then $d_\nu(P_1,P_2)=
1$.

\medskip

First, using the fact \eqref{dcp} that $G_0 G_0^* +\widetilde{G}_0^*\widetilde{G}_0
=I$, we obtain that
\begin{equation}
\label{eq_11}
G_1^* G_2= G_1^* G_0 G_0^* G_2 + G_1^* \widetilde{G}_0^* \widetilde{G}_0 G_2.
\end{equation}

\medskip

\noindent $\underline{1}^\circ$ Suppose that $\det (G_1^* G_2) \in
\inv S$ and $\iota(\det (G_1^* G_2))=\circ$. In this case,
$d_\nu(P_1,P_2)= \|\widetilde{G}_2 G_1\|_{\infty}$.  Using (S3) from
Proposition~\ref{prop_sv}, with
\begin{eqnarray*}
P&:=& \mathbf{G_1^* G_0 G_0^* G_2},\\
Q&:=&  \mathbf{G_1^* \widetilde{G}_0^* \widetilde{G}_0 G_2}.
\end{eqnarray*}
and \eqref{eq_11}, we have $ \underline{\sigma}( \mathbf{G_1^* G_0
  G_0^* G_2 } ) - \underline{\sigma}( \mathbf{G_1^* G_2} ) \leq
\overline{\sigma}( \mathbf{G_1^* \widetilde{G}_0^* \widetilde{G}_0
  G_2} ) $ pointwise on $\mathfrak{M}$. Furthermore, using (S4) and
(S5), and rearranging, we obtain
\begin{equation}
\label{eq_11.5}
\underline{\sigma}( {\mathbf{G_1^* G_2}})
\geq
\underline{\sigma}( {\mathbf{G_1^* G_0}}) \cdot
\underline{\sigma}( {\mathbf{G_0^* G_2}})
-
\overline{\sigma}( {\mathbf{\widetilde{G}_0 G_1}} ) \cdot
\overline{\sigma}(  {\mathbf{\widetilde{G}_0 G_2}} )
\end{equation}
pointwise on $ \mathfrak{M}$. Since $d_\nu(P_1, P_0)$ and $d_\nu (P_0,
P_2)$ are both in $[0,1]$, we can find $\alpha, \beta$ (which are maps
from $\mathfrak{M}$ to $[0, \frac{\pi}{2}]$) such that
\begin{eqnarray*}
\sin \alpha&=&\overline{\sigma}(\mathbf{\widetilde{G}_0 G_1}), \\
\sin \beta &=& \overline{\sigma}(\mathbf{\widetilde{G}_0 G_2}),
\end{eqnarray*}
pointwise on $\mathfrak{M}$.  Then using Lemma \ref{lemma_0.13}, it
follows from \eqref{eq_11.5} that
\begin{equation}
\label{eq_11.75}
\underline{\sigma}(\mathbf{G_1^* G_2})
\geq
(\cos \alpha) \cdot( \cos \beta) - (\sin \alpha)\cdot ( \sin \beta)
=
\cos (\alpha +\beta)
\end{equation}
pointwise on $\mathfrak{M}$.  Similarly, define
$\gamma:\mathfrak{M}\rightarrow [0, \frac{\pi}{2}]$ by $
\overline{\sigma}(\mathbf{\widetilde{G}_1 G_2}) =\sin \gamma, $ then
$\underline{\sigma}(\mathbf{G_1^* G_2}) =\cos \gamma$ pointwise on
$\mathfrak{M}$.  The inequality \eqref{eq_11.75} now says that $ \cos
\gamma \geq \cos (\alpha +\beta) $ pointwise on $\mathfrak{M}$.  Hence
\begin{eqnarray*}
\sin \gamma
&\leq&
\sin (\alpha+\beta)=
(\sin \alpha)\cdot( \cos \beta) + (\sin \beta)\cdot ( \cos \alpha)
\\
& \leq & (\sin \alpha)\cdot 1 + (\sin \beta)\cdot 1,
\end{eqnarray*}
that is, $ \overline{\sigma}(\mathbf{\widetilde{G}_1 G_2}) \leq
\overline{\sigma}(\mathbf{\widetilde{G}_0 G_1}) +
\overline{\sigma}(\mathbf{\widetilde{G}_0 G_2}) \leq
d_{\nu}(P_1,P_0)+d_\nu(P_0,P_2) $ pointwise on $\mathfrak{M}$.
Consequently, $ d_{\nu}(P_1, P_2) = \|\widetilde{G}_1 G_2\|_{\infty}
\leq d_{\nu}(P_1,P_0)+d_\nu(P_0,P_2)$.

\bigskip

\noindent $\underline{2}^\circ$ $\neg[\det (G_1^* G_2) \in \inv S
\textrm{ and }\iota(\det (G_1^* G_2))=\circ]$. In this case
$d_\nu(P_1, P_2) =1$. Let
\begin{eqnarray*}
  A&:=& G_1^* G_0 G_0^* G_2 , \textrm{ and}\\
  B&:=& G_1^* \widetilde{G}_0^* \widetilde{G}_0 G_2 .
\end{eqnarray*}
Using the fact that $G_{1}^* G_{0}$ and $G_{0}^* G_{2}$ are invertible
in $S^{m\times m}$, it follows also that $A$ is invertible in
$S^{m\times m}$.

Suppose that $\|A^{-1} B\|_{\infty}<1$. Then it follows from
\eqref{eq_11} that
$$
G_1^*G_2 =A+B= A (I+A^{-1}B)
$$
and so $G_1^* G_2$ is also invertible in $S^{m\times m}$. Consider the
map $H:[0,1] \rightarrow \inv S$, given by $ H(t)= \det (A (I+t A^{-1}
B))$, $t \in [0,1]$.  By Proposition \ref{prop_hom_inv},
\begin{eqnarray*}
\circ
&=&
\circ+\circ=\iota(G_1^* G_0)+ \iota(G_0^* G_2)
=\iota (\det A)
=
\iota(H(0))\\
&=&
\iota(H(1))= \iota( \det (G_1^*G_2)).
\end{eqnarray*}
But then we have that $\det (G_1^* G_2) \in \inv S $ and $\iota(\det
(G_1^* G_2))=\circ$, which is a contradiction.

So our assumption that $\|A^{-1} B\|_{\infty}<1$ cannot be true. From
the compactness of $\mathfrak{M}$ and the definition of the norm on
$\mC^{m\times m}$, it follows that there is a $\varphi \in
\mathfrak{M}$ such that $\overline{\sigma} ((\mathbf{A^{-1}
  B})(\varphi)) \geq 1$.  But then we have that
$$
1 \leq \overline{\sigma}((\mathbf{A^{-1} B})(\varphi)) \leq
\overline{\sigma}((\mathbf{A}(\varphi))^{-1})
\cdot
\overline{\sigma}(\mathbf{B}(\varphi)),
$$
and so
 \begin{equation}  \label{*}
 \underline{\sigma}(\mathbf{A}(\varphi)) \leq
\overline{\sigma}(\mathbf{B}(\varphi)).
\end{equation}
Thus
\begin{eqnarray*}
&&
(1- ( \overline{\sigma}( (\mathbf{\widetilde{G}_0 G_1})(\varphi)) )^2 )
\cdot
(1- ( \overline{\sigma}( (\mathbf{\widetilde{G}_0 G_2})(\varphi)) )^2 )
\\
&=&
( \underline{\sigma}( (\mathbf{G_1^* G_0})(\varphi) )^2
\cdot
( \underline{\sigma}( (\mathbf{G_0^* G_2})(\varphi) )^2
 \text{ by  Lemma \ref{lemma_0.13}}
\\
& \leq & \underline{\sigma}((G_1^* G_0 G_0^* G_2)(\varphi))^2 \text{ by (S5) in
 Proposition \ref{prop_sv}}
 \\
 & \leq &
 \overline{\sigma}( (G_1^* \widetilde{G}_0^* \widetilde{G}_0 G_2(\varphi))^2 \text{ by \eqref{*}}
 \\
&\leq&
( \overline{\sigma}( (\mathbf{\widetilde{G}_0 G_1})(\varphi)) )^2
\cdot
(  \overline{\sigma}( (\mathbf{\widetilde{G}_0 G_2})(\varphi)) )^2
\text{ by (S4) in Proposition \ref{prop_sv}.}
\end{eqnarray*}
With
\begin{eqnarray*}
x&:=& \overline{\sigma}( (\mathbf{\widetilde{G}_0 G_1})(\varphi)), \textrm{ and}\\
y&:=& \overline{\sigma}( (\mathbf{\widetilde{G}_0 G_2})(\varphi)),
\end{eqnarray*}
the above says that $(1-x^2) \cdot (1-y^2) \leq x^2 y^2$, and so $
1\leq x^2 +y^2$. Thus
$$
1\leq (\overline{\sigma}( (\mathbf{\widetilde{G}_0 G_1})(\varphi)))^2
+ (\overline{\sigma}( (\mathbf{\widetilde{G}_0 G_2})(\varphi)))^2 \leq
(d_\nu (P_0,P_1))^2 + (d_{\nu}(P_0, P_2))^2 .
$$
Consequently,
\begin{eqnarray*}
(d_\nu (P_0,P_1)\!+ \!d_{\nu}(P_0, P_2))^2
\geq
(d_\nu (P_0,P_1))^2 \!+\! (d_{\nu}(P_0, P_2))^2
\geq 1 =\!(d_\nu(P_1,P_2))^2.
\end{eqnarray*}
Taking square roots, we obtain the desired conclusion.

This completes the proof of the triangle inequality, and also the
proof of Theorem \ref{thm_d_nu_is_a_metric}.  \hfill$\Box$

\section{Robust stability theorem}
\label{section_robustness}

\noindent In this section we prove Theorem~\ref{theorem_rst}.

 \begin{definition}
   Given $P \in (\mF(R))^{p\times m}$ and $C\in (\mF(R))^{m\times p}$,
   we define the {\em stability margin} of the pair $(P,C)$ by
$$
\mu_{P,C}=\left\{ \begin{array}{ll}
\|H(P,C)\|_{\infty}^{-1} &\textrm{if }P \textrm{ is stabilized by }C,\\
0 & \textrm{otherwise.}
\end{array}\right.
$$
\end{definition}

The number $\mu_{P,C}$ can be interpreted as a measure of the
performance of the closed loop system comprising $P$ and $C$: larger
values of $\mu_{P,C}$ correspond to better performance, with
$\mu_{P,C}>0$ if $C$ stabilizes $P$.

\begin{proposition}
\label{prop_r_s_1}
If $P$ is stabilized by $C$, then $ \mu_{P,C}=\displaystyle
\inf_{\varphi\in \mathfrak{M}}
\underline{\sigma}(\mathbf{\widetilde{K}}(\varphi)\mathbf{G}(\varphi)).
$
\end{proposition}

\begin{proof} We now write $P = NM^{-1} = \widetilde{M}^{-1}
  \widetilde{N}$ for a normalized left/right coprime factorization of
  $P$ and $C = N_c M_c^{-1} = \widetilde{M}_c^{-1} \widetilde{N}_c$
  for a normalized left/right coprime factorization of $C$ and we set
$$
  G = \left[\begin{array}{cc} N \\ M \end{array}\right], \quad
\widetilde{G} = \left[\begin{array}{cc} -\widetilde{M} & \widetilde{N} \end{array}\right], \quad
  K = \left[\begin{array}{cc} N_c \\ M_c \end{array}\right], \quad
  \widetilde{K} = \left[\begin{array}{cc} - \widetilde{N}_c & \widetilde{M}_c \end{array}\right].
  $$
Then we have
\begin{eqnarray*}
  H(P,C)&=&
  \left[\begin{array}{cc} P \\ I \end{array} \right]
  (I-CP)^{-1} \left[\begin{array}{cc} -C & I \end{array} \right]\\
  &=&
  \left[\begin{array}{cc} NM^{-1} \\ I \end{array} \right]
  (I-\widetilde{M}_c^{-1} \widetilde{N}_c NM^{-1})^{-1}
  \left[\begin{array}{cc} -\widetilde{M}_c^{-1} \widetilde{N}_c & I \end{array} \right]\\
  &=&
  \left[\begin{array}{cc} N \\ M \end{array} \right]
  (-\widetilde{N}_c N+\widetilde{M}_c M)^{-1}
  \left[\begin{array}{cc} -\widetilde{N}_c & \widetilde{M}_c \end{array} \right]\\
  &=& G(\widetilde{K} G)^{-1} \widetilde{K}.
\end{eqnarray*}
Also, $ G^*G=\left[\begin{array}{cc} N^* & M^* \end{array} \right]
\left[\begin{array}{cc} N \\ M \end{array} \right] = N^*N+M^*M=I$.
Similarly,
$$
\widetilde{K} \widetilde{K}^*= \left[\begin{array}{cc}
    -\widetilde{N}_c & \widetilde{M}_c \end{array} \right]
\left[\begin{array}{cc} -\widetilde{N}_c^* \\ \widetilde{M}_c^*
  \end{array} \right] =
\widetilde{N}_c\widetilde{N}_c^*+\widetilde{M}_c\widetilde{M}_c^*=I.
$$
Using (S6) of Proposition~\ref{prop_sv}, we obtain for each $\varphi
\in \mathfrak{M}$ that
$$
\overline{\sigma}(\mathbf{G}(\varphi)(\mathbf{\widetilde{K}}(\varphi)
\mathbf{G}(\varphi))^{-1}
\mathbf{\widetilde{K}}(\varphi))=\overline{\sigma}((\mathbf{\widetilde{K}}(\varphi)
\mathbf{G}(\varphi))^{-1} ) =
\frac{1}{\underline{\sigma}(\mathbf{\widetilde{K}}(\varphi)
  \mathbf{G}(\varphi))}.
$$
Thus
$$
\frac{1}{\mu_{P,C}}= \sup_{\varphi\in \mathfrak{M}}
\overline{\sigma}(\mathbf{G}(\varphi)(\mathbf{\widetilde{K}}(\varphi)
\mathbf{G}(\varphi))^{-1} \mathbf{\widetilde{K}}(\varphi)) =
\sup_{\varphi\in
  \mathfrak{M}}\frac{1}{\underline{\sigma}(\mathbf{\widetilde{K}}(\varphi)
  \mathbf{G}(\varphi))},
$$ and so $\mu_{P,C}=\displaystyle\inf_{\varphi\in \mathfrak{M}}
\underline{\sigma}(\mathbf{\widetilde{K}}(\varphi)\mathbf{G}(\varphi))$.
\end{proof}

\begin{remark}  \label{R:mu<1}
 It is useful to note that
\begin{equation}  \label{mu<1}
\mu_{P,C} < 1
\end{equation}
for any $P$ and $C$ as above.  One way to see this is to note that
$H(P,C)$ is idempotent $H(P,C) \cdot H(P,C) = H(P,C)$; this forces $\|
H(P,C)\|_\infty \ge 1$.  Another way to see \eqref{mu<1} is to use the
formula for $\mu_{P,C}$ in Proposition \ref{prop_r_s_1} as follows.
Since $G^*G = I$ and $\widetilde{K} \widetilde{K}^* = i$, it follows
that $\overline{\sigma}(G) =1$ and $\overline{\sigma}(\widetilde{K}) =
1$.  Then it follows from various of the properties singular values
listed in Proposition \ref{prop_sv} that
$$
\underline{\sigma}(\widetilde{\mathbf K}(\varphi) {\mathbf G}(\varphi)) \le
\overline{\sigma} (\widetilde{\mathbf K}(\varphi) {\mathbf G}(\varphi)) \le
\overline{\sigma}(\widetilde{\mathbf K}(\varphi)) \cdot
\overline{\sigma}({\mathbf G}(\varphi)) = 1.
$$
\end{remark}

\begin{proposition}
The following are equivalent:
\begin{enumerate}
\item $C$ stabilizes $P$.
\item $\det (\mathbf{\widetilde{K}} (\varphi) \mathbf{G}(\varphi))\neq
  0$ for all $\varphi \in \mathfrak{M}$ and $\iota(\det(\widetilde{K}
  G))=\circ$.
\end{enumerate}
\end{proposition}
\begin{proof} Suppose that $C$ stabilizes $P$. Then from the
  calculation done above in the proof of Proposition~\ref{prop_r_s_1},
  we have
\begin{equation}
\label{eq_prop_r_s_2}
H(P,C)=G(\widetilde{K} G)^{-1} \widetilde{K}.
\end{equation}
But we know that $G$ is left invertible and $\widetilde{K}$ is right
invertible as matrices with entries from $R$. So from the above, we
see that $\widetilde{K} G\in R^{m\times m}$ is invertible as an
element of $R^{m\times m}$. In particular $\det (\widetilde{K} G)$ is
invertible as an element of $R$ and so $\det (\mathbf{\widetilde{K}}
(\varphi) \mathbf{G}(\varphi))\neq 0$ for all $\varphi \in
\mathfrak{M}$. Also from (A4), it follows that $
\iota(\det(\widetilde{K} G))=\circ$.

\medskip

Now suppose that $\det (\mathbf{\widetilde{K}} (\varphi)
\mathbf{G}(\varphi))\neq 0$ for all $\varphi \in \mathfrak{M}$ and
$\iota(\det(\widetilde{K} G))=\circ$. Then $\widetilde{K} G\in R \cap
\inv S$.  From (A4), we obtain that $\det (\widetilde{K} G)$ is
invertible as an element of $R$, and so we see from
\eqref{eq_prop_r_s_2} that $H(P,C)$ has entries from $R$. So $P$ is
stabilized by $C$.
\end{proof}

\begin{proposition}
$\mu_{P,C}=\mu_{C,P}$.
\end{proposition}
\begin{proof} It is not hard to see that $C$ stabilizes $P$ iff $P$
  stabilizes $C$.  We have
\begin{eqnarray*}
\mathbf{\widetilde{K}}\mathbf{\widetilde{G}^*}\mathbf{\widetilde{G}}\mathbf{\widetilde{K}^*}
+
\mathbf{\widetilde{K}}\mathbf{G}\mathbf{G^*}\mathbf{\widetilde{K}^*}
&=& I,\\
\mathbf{\widetilde{G}}\mathbf{\widetilde{K}^*}\mathbf{\widetilde{K}}\mathbf{\widetilde{G}^*}
+
\mathbf{\widetilde{G}}\mathbf{K}\mathbf{K^*}\mathbf{\widetilde{G}^*}
&=& I
\end{eqnarray*}
pointwise on $\mathfrak{M}$. So it follows from Lemma~\ref{lemma_0.12}
that
\begin{equation}
\label{eq_prop_4.1_16}
(\underline{\sigma}(\mathbf{\widetilde{K}}\mathbf{G}))^2
=
1-(\overline{\sigma}(\mathbf{\widetilde{K}}\mathbf{\widetilde{G}^*}))^2
=
(\underline{\sigma}(\mathbf{\widetilde{G}}\mathbf{K}))^2.
\end{equation}
This completes the proof.
\end{proof}

\begin{theorem}
\label{theorem_rst}
If $P_0, P_1 \in \mS(P,p,m)$ and $C\in \mS(R,m,p)$, then
$$
\sin^{-1} \mu_{P_1,C} \geq \sin^{-1} \mu_{P_0,C}-\sin^{-1} (d_\nu(P_0,P_1)).
$$
\end{theorem}

\begin{proof} If $d_{\nu}(P_0,P_1)\geq \mu_{P_0,C}$, then
  $\sin^{-1}(d_\nu(P_0,P_1))\geq \sin^{-1} \mu_{P_0,C}$ and so
  $\sin^{-1} \mu_{P_0,C} -\sin^{-1} (d_\nu(P_0,P_1))\leq 0$. The
  claimed inequality in the statement of the theorem now follows
  trivially since $\mu_{P_1,C}\geq 0$.

  We therefore assume in the rest of the proof that $d_\nu (P_0,P_1)
  <\mu_{P_0,C}$. As noted in Remark\ref{R:mu<1}, $\mu_{P_0,C} \leq 1$:  hence
  we must have
  $d_\nu(P_0,P_1)<1$. Also $\mu_{P_0,C}=0$ implies that $d_\nu
  (P_0,P_1)<0$, a contradiction to the fact that $d_\nu$ is a metric.
  Hence $\mu_{P_0,C}>0$, that is, $C$ stabilizes $P_0$. Now
$$
d_{\nu}(P_0,P_1) = \sup_{\varphi\in \mathfrak M}
\overline{\sigma}((\mathbf{\widetilde{G}_0 G_1})(\varphi)) < \inf_{
  \varphi\in \mathfrak M} \underline{\sigma}((\mathbf{\widetilde{K}
  G_0})(\varphi)) = \mu_{P_0,C},
$$
and so pointwise on $\mathfrak{M}$, there holds that $
\overline{\sigma}(\mathbf{\widetilde{G}_0 G_1}) <
\underline{\sigma}(\mathbf{\widetilde{K} G_0}).  $ But for numbers
$a,b\in (0,1)$,
$$
a<b\textrm{ iff }\displaystyle \frac{a^2}{1-a^2}<\frac{b^2}{1-b^2},
$$
and so  we have
$$
\frac{(\overline{\sigma}(\mathbf{\widetilde{G}_0 G_1}))^2}{1-
  (\overline{\sigma}(\mathbf{\widetilde{G}_0 G_1}))^2} <
\frac{(\underline{\sigma}(\mathbf{\widetilde{K} G_0}))^2}{1-
  (\underline{\sigma}(\mathbf{\widetilde{K} G_0}))^2}.
$$
Using Lemma~\ref{lemma_0.13} and \eqref{eq_prop_4.1_16}, we obtain
$\displaystyle \frac{\overline{\sigma}(\mathbf{\widetilde{G}_0
    G_1})}{\underline{\sigma} (\mathbf{G_0^* G_1})} <
\frac{\underline{\sigma}(\mathbf{\widetilde{K} G_0})}{
  \overline{\sigma}(\mathbf{\widetilde{K}\widetilde{G}_0^*})}$.  Thus
\begin{equation}
\label{eq_th_4.2_0}
\overline{\sigma}(\mathbf{\widetilde{K} \widetilde{G}_0^* \widetilde{G}_0 G_1} )
<
\underline{\sigma}(\mathbf{\widetilde{K} G_0 G_0^* G_1 }).
\end{equation}
But
\begin{equation}
\label{eq_th_4.2_1}
\widetilde{K}G_1=\widetilde{K}G_0G_0^*G_1+\widetilde{K}\widetilde{G}_0^* \widetilde{G}_0 G_1.
\end{equation}
Let $A:= \widetilde{K}G_0G_0^*G_1$, and
$B:=\widetilde{K}\widetilde{G}_0^* \widetilde{G}_0 G_1$.  Using the
fact that $\widetilde{K}G_0$ and $G_0^*G_1$ are invertible in
$S^{m\times m}$, it follows also that $A$ is invertible in $S^{m\times
  m}$. Also, from \eqref{eq_th_4.2_0}, it follows that $\|A^{-1}
B\|_{\infty}<1$.  Then it follows from \eqref{eq_th_4.2_1} that $
\widetilde{K}G_1 =A+B= A (I+A^{-1}B) $ and so $\widetilde{K}G_1$ is
also invertible in $S^{m\times m}$. Consider the map $H:[0,1]
\rightarrow \inv S$, defined by $ H(t)= \det (A (I+t A^{-1} B))$, $t
\in [0,1]$.  By Proposition \ref{prop_hom_inv}, it follows that
$H(0)=H(1)$, that is,
$$
\iota( \det (\widetilde{K}G_1) ) = \iota(
\det(\widetilde{K}G_0G_0^*G_1) ) = \iota( \det(\widetilde{K}G_0) )
+\iota( \det(G_0^*G_1) ) = \circ+\circ=\circ.
$$
But $\det (\widetilde{K}G_1)\in R$. By (A4) it follows that $\det
(\widetilde{K}G_1)$ is invertible as an element of $R$. Consequently
$C$ stabilizes $P_1$ and
$$
\mu_{P_1,C}=\inf_{\varphi\in \mathfrak{M}}
\underline{\sigma}((\mathbf{\widetilde{K} G_1})(\varphi)).
$$
From \eqref{eq_th_4.2_1}, we have
\begin{eqnarray*}
  \underline{\sigma} ( \mathbf{\widetilde{K} G_1} )
  &=&
  \underline{\sigma} ( \mathbf{\widetilde{K}G_0G_0^*G_1}
+\mathbf{\widetilde{K}\widetilde{G}_0^* \widetilde{G}_0 G_1} )
  \\
  &\geq &
  \underline{\sigma} ( \mathbf{\widetilde{K}G_0G_0^*G_1} )-
  \overline{\sigma}(\mathbf{\widetilde{K}\widetilde{G}_0^* \widetilde{G}_0 G_1} )
  \\
  &\geq & \underline{\sigma} ( \mathbf{\widetilde{K}G_0}) \underline{\sigma} (\mathbf{G_0^*G_1} )-
  \overline{\sigma}(\mathbf{\widetilde{K}\widetilde{G}_0^*})
  \overline{\sigma}(\mathbf{\widetilde{G}_0 G_1} )
  \\
  &=&
  \sin ( \sin^{-1}\underline{\sigma}(\mathbf{\widetilde{K} G_0}) -
  \sin^{-1}\overline{\sigma}(\mathbf{\widetilde{G}_0 G_1} ))  .
\end{eqnarray*}
Since $\sin^{-1}:[-1,1]\rightarrow [-\frac{\pi}{2}, \frac{\pi}{2}]$ is
an increasing function, it now follows that
$$
\sin^{-1} \underline{\sigma} ( \mathbf{\widetilde{K} G_1} )
\geq
\sin^{-1}\underline{\sigma}(\mathbf{\widetilde{K} G_0}) -
\sin^{-1}\overline{\sigma}(\mathbf{\widetilde{G}_0 G_1} ).
$$
Consequently, $\sin^{-1} \mu_{P_1,C} \geq \sin^{-1}
\mu_{P_0,C}-\sin^{-1} (d_\nu(P_0,P_1))$.
\end{proof}

\begin{corollary}
If $P_0,P\in \mS(R,p,m)$, then
$$
\mu_{P,C} \geq \mu_{P_0,C}-d_{\nu}(P_0,P).
$$
\end{corollary}

\begin{proof} For $x,y,z\in [0,1]$, if
 $
\sin^{-1} x\leq \sin^{-1} y+\sin^{-1} z.
 $
By taking the cosine of both sides and using that the $\cos$ is a decreasing function on
 $[0, \frac{\pi}{2}]$, we then get
 $
\sqrt{1-x^2} \geq \sqrt{1-y^2} \sqrt{1-z^2}-yz$,
which in turn implies that
 $$
 (\sqrt{1-x^2}+yz)^2\geq (1-y^2)(1-z^2).
$$
Hence
$
x^2 \leq y^2 +z^2+2yz\sqrt{1-x^2}\leq y^2 +z^2+2yz\cdot 1=(y+z)^2$,
which gives finally that $ x\leq y+z$.  The claimed inequality now
follows immediately from the inequality in Theorem~\ref{theorem_rst}
upon setting $x = \mu_{P_0,C}$, $y = d_\nu(P_0,P)$ and $z =
\mu_{P,C}$.
\end{proof}

The above result says that if the controller $C$ performs sufficiently
well with the nominal plant $P_0$, and the distance $d_{\nu}(P_0,P)$
between the plant $P$ and $P_0$ is sufficiently small, then $C$ is
guaranteed to achieve a certain level of performance with the plant
$P$. So if $P$ and $P_0$ represent alternate models of the system (one
which is ``true'' and one which is our nominal model) and if
$d_\nu(P_0,P)$ is small, then the difference between $P$ and $P_0$ can
be ignored for the purposes of designing a stabilizing controller.

Another way of stating the result in Theorem~\ref{theorem_rst} is that
if $C$ stabilizes $P_0$ with a stability margin $\mu_{P,C}>m$, and $P$
is another plant which is close to $P_0$ in the sense that
$d_\nu(P,P_0)\leq m$, then $C$ is also guaranteed to stabilize $P$.
Furthermore, if $C$ satisfies the stronger condition $\mu_{P,C}>M>m$
for a number $M$, then $C$ is also guaranteed to stabilize $P$ {\em
  with} a stability margin $\mu_{P,C}$ which satisfies $\mu_{P,C}\geq
\sin^{-1}M-\sin^{-1}m$.

\section{Applications}
\label{section_applications}

\noindent Now we specialize $R$ to several classes of stable transfer functions
and obtain various extensions of the $\nu$-metric. Some of the verifications
of the properties (A1)-(A4) are similar to the section on examples from \cite{Sas}.

\subsection{The disk algebra}
\label{subsection_disk_algebra}
Let
$$
\mD:= \{ z\in \mC: |z| <1\},\quad
\overline{\mD}:= \{ z\in \mC: |z| \leq 1\}, \quad
\mT:= \{ z\in \mC: |z|=1\}.
$$
The {\em disk algebra} $A(\mD)$ is the set of all functions $f:
\overline{\mD} \rightarrow \mC$ such that $f$ is holomorphic in $\mD$
and continuous on $\overline{\mD}$.  Let $C(\mT)$ denote the set of
complex-valued continuous functions on the unit circle $\mT$. For each
$f\in \inv C(\mT)$, we can define the {\em winding number} ${\tt
  w}(f)\in \mZ$ of $f$ as follows:
$$
{\tt w}(f)= \frac{1}{2\pi}(\Theta(2\pi)- \Theta(0)),
$$
where $\Theta:[0,2\pi] \rightarrow \mR$ is a continuous function such
that
$$
f(e^{it}) =|f(e^{it})| e^{i \Theta(t)}, \quad t \in [0,2\pi].
$$
The existence of such a $\Theta$ can be proved; see \cite[Lemma
4.6]{Ull}. Also, it can be checked that ${\tt w}$ is well-defined and
integer-valued. Geometrically, ${\tt w}(f)$ is the number of times the
curve $t \mapsto f(e^{it}): [0,2\pi] \rightarrow \mC$ winds around the
origin in a counterclockwise direction.

Recall the definition of a full subring.

\begin{definition}
  Let $R_1, R_2$ be commutative unital rings, and let $R_1$ be a
  subring of $R_2$. Then $R_1$ is said to be a {\em full} subring of $R_2$
  if for every $x\in R_1$ such that $x$ is invertible in $R_2$, it
  holds that $x$ is invertible in $R_1$.
\end{definition}

\begin{lemma}
\label{lemma_disk_algebra}
  Let
\begin{eqnarray*}
  R&=& \textrm{\em a unital full subring of }A(\mD),\\
  S&:=& C(\mT), \\
  G&:=& \mZ, \\
  \iota&:=&{\tt w}.
\end{eqnarray*}
Then {\em (A1)-(A4)} are satisfied.
\end{lemma}
\begin{proof} (A1) is clear. The involution $\cdot^{*}$ in (A2) is
  defined by $f^*(z)=\overline{f(z)}$ ($z\in \mT$) for $f\in C(\mT)$.
  (A3)(I1) and (A3)(I2) are evident from the definition of {\tt w}.
  Also, the map ${\tt w}: \inv C(\mT) \rightarrow \mZ$ is locally
  constant (that is, it is continuous when $\mZ$ is equipped with the
  discrete topology and $C(\mT)$ is equipped with the usual
  $\sup$-norm); see \cite[Lemma~4.6.(ii)]{Ull}.  So (A3)(I3) holds as
  well.  Finally, we will show below that (A4) holds.

Suppose that $f \in R \cap (\inv C(\mT))$ is invertible as an element
of $R$. Then obviously $f$ is also invertible as an element of
$A(\mD)$.  Hence it has no zeros or poles in $\overline{\mD}$.  For
$r\in (0,1)$, define $f_r \in A(\mD)$ by $f_r(z)=f(rz)$ ($z\in
\overline{\mD}$).  Then $f_r$ also has no zeros or poles in
$\overline{\mD}$, and has a holomorphic extension across $\mT$. From
the Argument Principle (applied to $f_r$), it follows that ${\tt
  w}(f_r)=0$. But $\|f_r-f\|_\infty \rightarrow 0$ as $r\nearrow 1$.
Hence ${\tt w}(f)= \displaystyle \lim_{r\rightarrow 1} {\tt w}(f_r)=
\displaystyle \lim_{r\rightarrow 1} 0=0$.

Suppose, conversely, that $f\in R \cap (\inv C(\mT))$ is such that
${\tt w}(f)=0$. For all $r\in (0,1)$ sufficiently close to $1$, we
have that $f_r\in \inv C(\mT)$. Also, by the local constancy of ${\tt
  w}$, for $r$ sufficiently close to $1$, ${\tt w}(f_r)= {\tt
  w}(f)=0$. By the Argument principle, it then follows that $f_r$ has
no zeros in $\overline{\mD}$. Equivalently, $f$ has no zeros in $r
\overline{\mD}$. But letting $r\nearrow 1$, we see that $f$ has no
zeros in $\mD$. Moreover, $f$ has no zeros on $\mT$ either, since
$f\in \inv C(\mT)$. Thus $f$ has no zeros in $\overline{\mD}$.
Consequently, we conclude that $f$ is invertible as an element of
$A(\mD)$. (Indeed, $f$ is invertible as an element of
$C(\overline{\mD})$, and it is also then clear that this inverse is
holomorphic in $\mD$.) Finally, since $R$ is a full subring of
$A(\mD)$, we can conclude that $f$ is invertible also as an element of
$R$.
\end{proof}

\noindent  Besides $A(\mD)$ itself, some other examples of such $R$ are:
\begin{enumerate}
\item $RH^\infty(\mD)$, the set of all rational functions without
  poles in $\overline{\mD}$.
\item The Wiener algebra $W^{+}(\mD)$ of all functions $f\in A(\mD)$
  that have an absolutely convergent Taylor series about the origin:

$\displaystyle \sum_{n=0}^\infty |f_n|<+\infty$, where $f(z)=
\displaystyle \sum_{n=0}^\infty f_n z^n$ ($z\in \mD$).
\item $\partial^{-n}H^\infty(\mD)$, the set of $f:\mD \rightarrow \mC$ such
  that $f,f^{(1)},f^{(2)},\dots,f^{(n)}$ belong to $H^\infty(\mD)$. Here
  $H^\infty(\mD)$ denotes the Hardy algebra of all bounded and holomorphic
  functions on $\mD$.
\end{enumerate}

\noindent In the definition of the $\nu$-metric given in
Definition~\ref{def_nu_metric} corresponding to
Lemma~\ref{lemma_disk_algebra}, the $\|\cdot\|_\infty$ now means the
following: if $F\in (C(\mT))^{p\times m}$, then
$$
\|F\|_\infty =\max_{z\in \mT} \nm F(z)\nm .
$$
This follows from \eqref{norm}, since the maximal ideal space
$\mathfrak{M}$ of $S=C(\mT)$ can be identified with the unit circle as
a topological space; see \cite[Example~11.13.(a)]{Rud91}.

\begin{remark}
  $RH^\infty(\mD)$ is a projective free ring since it is a B\'ezout
  domain. Also $A(\mD)$, $W^+(\mD)$, or $\partial^{-n}H^\infty(\mD)$
  are projective free rings, since their maximal ideal space is
  $\overline{\mD}$, which is contractible; see \cite{BruSas}. Thus if
  $R$ is one of $RH^\infty(\mD)$, $A(\mD)$, $W^{+}(\mD)$ or
  $\partial^{-n}H^\infty(\mD)$, then the set $\mS(R, p, m)$ of plants
  possessing a left and a right coprime factorization coincides with
  the class of plants that are stabilizable by
  \cite[Theorem~6.3]{Qua}.
\end{remark}

\subsection{Almost periodic functions}
\label{subsection_AP}

The algebra $AP$ of complex valued (uniformly) {\em almost periodic
  functions} is the smallest closed subalgebra of $L^\infty(\mR)$ that
contains all the functions $e_\lambda := e^{i \lambda y}$. Here the
parameter $\lambda$ belongs to $\mR$.  For any $f\in AP$, its {\em
  Bohr-Fourier series} is defined by the formal sum
\begin{equation}
\label{eq_BFs}
\sum_{\lambda} f_\lambda e^{i  \lambda y} , \quad y\in \mR,
\end{equation}
where
$$
f_\lambda:= \lim_{N\rightarrow \infty} \frac{1}{2N}
\int_{[-N,N]}   e^{-i \lambda y} f(y)dy, \quad
\lambda \in \mR,
$$
and the sum in \eqref{eq_BFs} is taken over the set $
\sigma(f):=\{\lambda \in \mR\;|\; f_\lambda \neq 0\}$, called the {\em
  Bohr-Fourier spectrum} of $f$. The Bohr-Fourier spectrum of every
$f\in AP$ is at most a countable set.

The {\em almost periodic Wiener algebra} $APW$ is defined as the set
of all $AP$ such that the Bohr-Fourier series \eqref{eq_BFs} of $f$
converges absolutely. The almost periodic Wiener algebra is a Banach
algebra with pointwise operations and the norm $\|f\|:=
\displaystyle\sum_{\lambda \in \mR} |f_\lambda|$.  Set
\begin{eqnarray*}
  AP^+&=& \{ f \in AP \;|\; \sigma(f) \subset [0,\infty)\}\\
  APW^+&=& \{ f \in APW \;|\; \sigma(f) \subset [0,\infty)\}.
\end{eqnarray*}
Then $AP^+$ (respectively $APW^+$) is a Banach subalgebra of $AP$
(respectively $APW$). For each $f\in \inv AP$, we can define the {\em
  average winding number} $w(f)\in \mR$ of $f$ as follows:
$$
w(f)= \lim_{T \rightarrow \infty} \frac{1}{2T}
\bigg( \arg (f(T))-\arg(f(-T))\bigg).
$$
See \cite[Theorem 1, p. 167]{JesTor}.

\begin{lemma}
\label{lemma_AP_111}
Let
\begin{eqnarray*}
  R&:=& \textrm{\em a unital full subring of }AP^+\\
  S&:=& AP, \\
  G&:=& \mR,\\
  \iota&:=& w.
\end{eqnarray*}
Then {\em (A1)-(A4)} are satisfied.
\end{lemma}
\begin{proof} (A1) is clear. The involution $\cdot^{*}$ used in (A2) is defined by
$$
f^*(y)=\overline{f(y)}, \quad y \in \mR,
$$
for $f\in AP$.  (A3)(I1) and (A3)(I2) follow from the
  definition of $w$. (A3)(I3) follows for example from \cite[Theorem~2.6 and Example~2.10]{Mur}, where
it is shown that $w$ is a {\em topological index} on $AP$, and hence in particular, it is
locally constant.

Finally, (A4) follows from
  \cite[Theorem~1,~p.776]{CalDes} which says that $f\in AP^+$
  satisfies
\begin{equation}
\label{cc_AP+}
\inf_{\textrm{Im}(s)\geq 0} |f(s)| >0
\end{equation}
iff $\displaystyle \inf_{y\in \mR}|f(y)| >0$ and $w(f)=0$.  But
$$
\displaystyle \inf_{y\in \mR}|f(y)| >0
$$
is equivalent to $f$ being an invertible element of $AP$ by the corona
theorem for $AP$ (see for example \cite[Exercise~18,~p.24]{Gam}). Also
the equivalence of \eqref{cc_AP+} with that of the invertibility of
$f$ as an element of $AP^+$ follows from the Arens-Singer corona
theorem for $AP^+$ (see for example \cite[Theorems~3.1,~4.3]{Bot}).
Finally, the invertibility of $f \in R$ in $R$ is equivalent to the
invertibility of $f$ as an element of $AP^+$ since $R$ is a full
subring of $AP^+$.
\end{proof}

\begin{remark}
\label{examples_of_R_in_AP+}
Specific examples of such $R$ are $AP^+$ and $APW^+$. More generally,
let $\Sigma \subset [0,+\infty)$ be an {\em additive semigroup} (if
$\lambda,\mu \in \Sigma $, then $\lambda+\mu \in \Sigma$) and suppose
$0 \in \Sigma$.  Denote
\begin{eqnarray*}
  AP_\Sigma&=& \{ f \in AP \;|\; \sigma(f) \subset \Sigma\}\\
  APW_\Sigma&=& \{ f \in APW \;|\; \sigma(f) \subset \Sigma\}.
\end{eqnarray*}
Then $AP_\Sigma$ (respectively $APW_\Sigma$) is a unital Banach
subalgebra of $AP^+$ (respectively $APW^+$). Let
$\overline{Y_{\Sigma}}$ denote the set of all maps $\theta:\Sigma
\rightarrow [0,+\infty]$ such that $ \theta(0)=0$ and $\theta
(\lambda+\mu)= \theta (\lambda)+ \theta(\mu)$ for all $\lambda, \mu
\in \Sigma$.  Examples of such maps $\theta$ are the following. If
$y\in [0,+\infty)$, then $\theta_y$, defined by $\theta_y(\lambda)=
\lambda y$, $\lambda \in \Sigma$, belongs to $\overline{Y_{\Sigma}}$.
Another example is $\theta_{\scriptscriptstyle \infty}$, defined as follows:
$$
\theta_{\scriptscriptstyle \infty}(\lambda)= \left\{ \begin{array}{ll}
0 & \textrm{if } \lambda =0,\\
+\infty & \textrm{if } \lambda \neq 0.
\end{array}\right.
$$
So in this way we can consider $[0,+\infty]$ as a subset of
$\overline{Y_{\Sigma}}$.

The results \cite[Proposition~4.2, Theorem 4.3]{Bot} say that if
$\overline{Y_{\Sigma}} \subset [0,+\infty]$, and $f$ belongs to $ AP_\Sigma$
(respectively to $APW_\Sigma$), then $f$ belongs to $ \inv AP_\Sigma$ (respectively
 to $ \inv APW_\Sigma$) iff \eqref{cc_AP+} holds. So in this case
$AP_\Sigma$ and $APW_\Sigma$ are unital full subalgebras of $AP^+$.
\end{remark}

In the definition of the $\nu$-metric given in
Definition~\ref{def_nu_metric} corresponding to
Lemma~\ref{lemma_AP_111}, the $\|\cdot\|_\infty$ now means the
following: if $F\in (AP)^{p\times m}$, then
$$
\|F\|_\infty =\sup_{y\in \mR} \nm F(y)\nm .
$$
This follows from \eqref{norm}, since $\mR$ is dense in the maximal
ideal space $\mathfrak{M}$ (which is the Bohr compactification $\mR_B$
of $\mR$) of the Banach algebra $S=AP$; see \cite[Exercise 18, p.24]{Gam}.

\begin{remark}
  It was shown in \cite{BruSas} that $AP^+$ and $APW^+$ are projective
  free rings.  Thus if $R=AP^+$ or $APW^+$, then the set $\mS(R, p,
  m)$ of plants possessing a left and a right coprime factorization
  coincides with the class of plants that are stabilizable by
  \cite[Theorem~6.3]{Qua}.
\end{remark}

\subsection{Algebras of Laplace transforms of measures without a
  singular nonatomic part}  \label{subsection_CD}

Let $\mC_{+}:=\{s\in \mC\;|\; \textrm{Re}(s)\geq 0\}$ and let
$\calA^+$ denote the Banach algebra

$
\calA^+=\left\{ s (\in \mC_{+}) \mapsto \widehat{f_a}(s)
  +\displaystyle \sum_{k=0}^\infty f_k e^{- s t_k} \; \bigg| \;
\begin{array}{ll}
f_a \in L^{1}(0,\infty), \;(f_k)_{k\geq 0} \in \ell^{1},\\
0=t_0 <t_1 ,t_2 , t_3, \dots
\end{array} \right\}
$

\noindent equipped with pointwise operations and the norm:

$
\|F\|=\|f_a\|_{\scriptscriptstyle L^{1}} +
\|(f_k)_{k\geq 0}\|_{\scriptscriptstyle \ell^1},
\;\; F(s)=\widehat{f_a}(s) +\displaystyle\sum_{k=0}^\infty f_k e^{-st_k}\;\;(s\in \mC_+).
$

\noindent Here $\widehat{f_a}$ denotes the {\em Laplace transform of}
$f_a$, given by
$$
\widehat{f_a}(s)=\displaystyle \int_0^\infty e^{-st} f_a(t)
dt, \quad s \in \mC_+.
$$
Similarly, define the Banach algebra $\calA$ as follows (\cite{GohFel}):

$
\!\!\!\!\!\!\calA\!=\!\left\{ iy (\in i\mR) \mapsto \widehat{f_a}(iy)
  +\!\!\!\displaystyle\sum_{k=-\infty}^\infty f_k e^{- iy t_k} \bigg|
\begin{array}{ll}
f_a \in L^{1}(\mR), \;(f_k)_{k\in \mZ } \in \ell^{1},\\
\dots, t_{-2}, t_{-1}<\!0\!=\!t_0\! <t_1 ,t_2 ,  \dots
\end{array} \!\!\!\right\}
$

\noindent equipped with pointwise operations and the norm:

$
\|F\|=\|f_a\|_{\scriptscriptstyle L^{1}} + \|(f_k)_{k\in
  \mZ}\|_{\scriptscriptstyle \ell^1}, \;\; F(iy):=\widehat{f_a}(iy)
+\displaystyle\sum_{k=-\infty}^\infty f_k e^{-iy t_k}\;\;(y\in \mR).
$

\noindent Here $\widehat{f_a}$ is the {\em Fourier transform of}
$f_a$, $
\widehat{f_a}(iy)= \displaystyle \int_{-\infty}^\infty e^{-iyt} f_a(t)
dt$, ($y \in \mR$).

\noindent It can be shown that $\widehat{L^{1}(\mR)}$ is an ideal of $\calA$.

\noindent For $ F=\widehat{f_a} +\displaystyle\sum_{k=-\infty}^\infty
f_k e^{-i\cdot t_k}\in \calA$, we set
$F_{AP}(iy)=\displaystyle\sum_{k=-\infty}^\infty f_k e^{-iy t_k}$
($y\in \mR$).

\noindent If $F =\widehat{f_a}+F_{AP} \in \inv \calA$, then it can be shown that
$F_{AP}(i \cdot) \in \inv AP$ as follows. First of all, the maximal
ideal space of $\calA$ contains a copy of the maximal ideal space of
$APW$ in the following manner: if $\varphi \in M(APW)$, then the map
$\Phi: \calA \rightarrow \mC$ defined by
$\Phi(F)=\Phi(\widehat{f_a}+F_{AP})=\varphi(F_{AP}(i\cdot))$, ($F \in
\calA$), belongs to $M(\calA)$. So if $F$ is invertible in $\calA$, in
particular for every $\Phi$ of the type describe above, $0 \neq
\Phi(F)=\varphi(F_{AP}(i\cdot))$. Thus by the elementary theory of
Banach algebras, $F_{AP}(i\cdot)$ is an invertible element of $AP$.

Moreover, since $\widehat{L^1(\mR)}$ is an ideal in $\calA$,
$F_{AP}^{-1}\widehat{f_a}$ is the Fourier transform of a function in
$L^{1}(\mR)$, and so the map $y \mapsto
1+(F_{AP}(iy))^{-1}\widehat{f_a}(iy)=\frac{F(iy)}{F_{AP}(iy)}$ has a
well-defined winding number ${\tt w}$ around $0$. Define $W: \inv
\calA \rightarrow \mR\times \mZ$ by
\begin{equation}   \label{Windex}
 W (F)= (w(F_{AP}), {\tt
  w}(1+F_{AP}^{-1} \widehat{f_a}) ),
\end{equation}
 where
$F=\widehat{f_a}+F_{AP} \in \inv \calA$, and
$$
\begin{array}{ll}
w(F_{AP})
:=
\displaystyle \lim_{R \rightarrow \infty} \frac{1}{2R}
\bigg( \arg \big(F_{AP}(iR)\big)-\arg\big(F_{AP}(-iR)\big)\bigg),
\\
{\tt w}(1+F_{AP}^{-1} \widehat{f_a})
:=
\displaystyle \frac{1}{2\pi}\bigg( \arg \big(1+(F_{AP}(iy)\big)^{-1} \widehat{f_a}(iy)  )
\bigg|_{y=-\infty}^{y=+\infty}\bigg).
\end{array}
$$

\begin{lemma}
\label{lemma_inv_A}
$F=\widehat{f_a}+F_{AP} \in \calA$ is invertible iff for all $y\in
\mR$, $F(iy) \neq 0$ and $\displaystyle \inf_{y\in \mR} |F_{AP}(iy)| >0$ .
\end{lemma}
\begin{proof} The `only if' part is clear. We simply show the `if'
  part below.

  Let $F=\widehat{f_a}+F_{AP} \in \calA$ be such that $F(iy) \ne 0$ for all $y \in {\mathbb R}$ and
$$
\inf_{y\in \mR} |F_{AP}(iy)| >0.
$$
Thus $F_{AP}(i\cdot)$ is invertible as an element of $AP$. Hence
$F=F_{AP}(1+\widehat{f_{a}} F_{AP}^{-1})$ and so it follows that $ (
1+\widehat{f_{a}} F_{AP}^{-1})(iy) \neq 0$ for all $y \in \mR$.  But
by the corona theorem for
$$
\calW:=\widehat{L^1(\mR)}+\mC
$$
(see \cite[Corollary~1,~p.109]{GelRaiShi}), it follows that
$1+\widehat{f_{a}} F_{AP}^{-1}$ is invertible as an element of $\calW$
and in particular, also as an element of $\calA$. This completes the
proof.
\end{proof}

\begin{lemma}
\label{lemma_calA}
Let
\begin{eqnarray*}
  R&:=&\textrm{\em a unital full subring of } \calA^+, \\
  S&:=& \calA, \\
  G&:=& \mR\times \mZ,\\
  \iota&:=& W.
\end{eqnarray*}
Then {\em (A1)-(A4)} are satisfied.
\end{lemma}
\begin{proof} (A1) is clear. The involution $\cdot^*$ in (A2) is defined by
$$
F^*(iy)=\overline{F(iy)}, \quad y \in \mR,
$$
for $F\in \calA$.  (A3)(I2) is now easy to see from the definition of $W$.
Also, (A3)(I1) follows from the
  definition of $W$ as follows. Let $F=\widehat{f_a}+F_{AP}$ and
  $G=\widehat{g_a}+G_{AP}$. Then we have
$$
w(F_{AP}G_{AP})=w(F_{AP})+w(G_{AP})
$$
from the definition of $w$. Thus
\begin{eqnarray*}
W(FG)&=&W(( \widehat{f_a}+F_{AP})(\widehat{g_a}+G_{AP}))\\
&=& W(\widehat{f_a}\widehat{g_a}+\widehat{f_a}G_{AP}+ \widehat{g_a}F_{AP}+ F_{AP}G_{AP})\\
&=& ({\tt w}(1+(F_{AP}G_{AP})^{-1}(\widehat{f_a}\widehat{g_a}
+\widehat{f_a}G_{AP}+ \widehat{g_a}F_{AP}), w(F_{AP}G_{AP}))\\
&=& ({\tt w}((1+F_{AP}^{-1} \widehat{f_a})(1+G_{AP}^{-1} \widehat{g_a})) , w(F_{AP})+w(G_{AP}))\\
&=& ({\tt w}(1+F_{AP}^{-1} \widehat{f_a})+{\tt w}(1+G_{AP}^{-1} \widehat{g_a}),w(F_{AP})+w(G_{AP}))\\
&=& W(\widehat{f_a}+F_{AP})+W(\widehat{g_a}+G_{AP}).
\end{eqnarray*}
So (A3)(I2) holds.

The local constancy of $W$ demanded in (A3)(I3) can be seen in the following manner.
We have already noted that $w$ is locally constant on $\inv AP$ and ${\tt w}$ is
locally constant on $\inv C(\mT)$. Note that ${\tt w} (1+F_{AP}^{-1} \widehat{f_a})$ defined above
 is just ${\tt w}(\varphi)$ where
$$
\varphi(\theta)= (1+F_{AP}^{-1} \widehat{f_a})(iy),
\textrm{ where } iy=\frac{1+e^{i\theta}}{1-e^{i\theta}}, \quad \theta \in(0,2\pi).
$$
Hence (A3)(I3) follows.

Finally we check that (A4) holds. Suppose that $
F=\widehat{f_{a}}+F_{AP}$ belonging to $ \calA^{+} \cap (\inv \calA)$,
is such that $W(F)=0$. Since $F$ is invertible in $\calA$, it follows
that $F_{AP}(i\cdot)$ is invertible as an element of $AP$. But
$w(F_{AP})=0$, and so $F_{AP}(i\cdot)\in AP^+$ is invertible as an
element of $AP^+$. But this implies that $1+F_{AP}^{-1}
\widehat{f_{a}}$ belongs to the Banach algebra
$$
\calW^{+}:=\widehat{L^{1}(0,\infty)}+\mC.
$$
Moreover, it is bounded away from $0$ on $i\mR$ since
$$
1+F_{AP}^{-1} \widehat{f_{a}} =\frac{F}{F_{AP}},
$$
and $F$ is bounded away from zero on $i\mR$. Moreover ${\tt
  w}(1+F_{AP}^{-1} \widehat{f_{a}})=0$, and so it follows that
$1+F_{AP}^{-1} \widehat{f_{a}}$ is invertible as an element of
$\calW^{+}$, and in particular in $\calA^+$. Since
$F=(1+F_{AP}^{-1} \widehat{f_a})F_{AP}$ and we have shown that both
$(1+F_{AP}^{-1} \widehat{f_a})$ as well as $F_{AP}$ are invertible as
elements of $\calA^+$, it follows that $F$ is invertible in $\calA^+$.
\end{proof}

An example of such a $R$ (besides $\calA^+$) is the algebra
$$
\widehat{L^{1}(0,+\infty)}+APW_\Sigma(i\cdot):=\{\widehat{f_a}+F_{AP}:
f_a\in L^{1}(0,+\infty), \; F_{AP}(i\cdot)\in APW_\Sigma \},
$$
where $\Sigma$ is as
described in Remark~\ref{examples_of_R_in_AP+}.

In the definition of the $\nu$-metric given in
Definition~\ref{def_nu_metric} corresponding to
Lemma~\ref{lemma_calA}, the $\|\cdot\|_\infty$ now means the
following: if $F\in \calA^{p\times m}$, then
$$
\|F\|_\infty =\sup_{y\in \mR} \nm F(iy)\nm .
$$
This follows from \eqref{norm}, since $\mR$ is dense in the maximal
ideal space $\mathfrak{M}$ of the Banach algebra $S=\calA$; see
\cite[Theorems 4.20.1 and 4.20.4]{HilPhi}.

\begin{remark}
  It was shown in \cite{BruSas} that $\calA^+$ is a projective free
  ring.  Thus the set $\mS(\calA^+, p, m)$ of plants possessing a left
  and a right coprime factorization coincides with the class of plants
  that are stabilizable by \cite[Theorem~6.3]{Qua}.
\end{remark}

\subsection{The polydisk algebra}  \label{subsection_polydisk}

Let
\begin{eqnarray*}
\mD^n&:=& \{ (z_1, \dots, z_n)\in \mC^n: |z_i| <1 \textrm{ for } i=1,\dots, n\},\\
\overline{\mD^n}&:=& \{ (z_1, \dots, z_n)\in \mC^n: |z_i| \leq 1\textrm{ for } i=1,\dots, n\},\\
\mT^n&:=& \{ (z_1, \dots, z_n)\in \mC^n: |z_i|=1\textrm{ for } i=1,\dots, n\}.
\end{eqnarray*}
The {\em polydisk algebra} $A(\mD^n)$ is the set of all functions $f:
\overline{\mD^n} \rightarrow \mC$ such that $f$ is holomorphic in $\mD^n$
and continuous on $\overline{\mD^n}$.

If $f\in A(\mD^n)$, then the function $f_{d}$ defined by $z\mapsto
f(z,\dots, z): \overline{\mD} \rightarrow \mC$ belongs to the disk
algebra $A(\mD)$, and in particular also to $C(\mT)$. The map
$$
f\mapsto (f|_{\mT^n}, f_d): A(\mD^n) \rightarrow C(\mT^n) \times C(\mT)
$$
is a ring homomorphism. This map is also injective, and this is an
immediate consequence of Cauchy's formula; see \cite[p.4-5]{Rud69}.
We recall the following result; see \cite[Theorem~4.7.2, p.87]{Rud69}.

\begin{proposition}
\label{Rud_prop}
  Suppose that $\Psi=(\psi_1, \dots, \psi_n)$ is a continuous map from
  $\overline{\mD}$ into $\overline{\mD^n}$, which carries $\mT$ into
  $\mT^n$ and the winding number of each $\psi_i$ is positive. Then
  for every $f\in A(\mD^n)$, $f(\Psi(\overline{\mD}) \cup \mT^n)=
  f(\overline{\mD^n})$.
\end{proposition}

\begin{lemma}
\label{lemma_polydisk_algebra}
  Let
\begin{eqnarray*}
  R&=& \textrm{\em a unital full subring of }A(\mD^n),\\
  S&:=& C(\mT^n) \times C(\mT), \\
  G&:=& \mZ, \\
  \iota&:=&((g, h) \mapsto {\tt w}(h)) .
\end{eqnarray*}
Then {\em (A1)-(A4)} are satisfied.
\end{lemma}
\begin{proof} (A1) is clear. The involution $\cdot^*$ in (A2) is
  defined as follows: if $(f,g)\in C(\mT^n) \times C(\mT)$, then
  $(f,g)^*:=(f^*, g^*)$, where
\begin{eqnarray*}
f^*(z_1,\dots, z_n)&=&\overline{f(z_1,\dots, z_n)}, \quad  (z_1,\dots, z_n)\in \mT^n, \\
g^*(z)&=&\overline{g(z)}, \quad  z\in \mT.
\end{eqnarray*}
(A3) was proved earlier in Subsection~\ref{subsection_disk_algebra}.
Finally, we will show below that (A4) holds, following
\cite{DecMurSae}.

  Suppose that $f\in A(\mD^n)$ is such that $f|_{\mT^n} \in \inv
  C(\mT^n)$, $f_d \in \inv C(\mT)$ and that ${\tt w}(f_d)=0$. We use
  Proposition~\ref{Rud_prop}, with $\Psi(z):=(z,\dots, z)$ ($z\in
  \overline{\mD}$). Then we know that $f$ will have no zeros in
  $\overline{\mD^n}$ if $f(\Psi (\overline{\mD}))$ does not contain
  $0$. But since $f_d \in \inv C(\mT)$ and ${\tt w}(f_d)=0$, it
  follows that $f_d$ is invertible as an element of $A(\mD)$ by the
  result in Subsection~\ref{subsection_disk_algebra}. But this implies
  that $f(\Psi (\overline{\mD}))$ does not contain $0$.

  Now suppose that $f\in A(\mD^n)$ with $f|_{\mT^n} \in \inv
  C(\mT^n)$, $f_d \in \inv C(\mT)$, and that it is invertible as an
  element of $A(\mD^n)$. But then in particular, $f_d$ is an
  invertible element of $A(\mD)$, and so again by the result in
  Subsection~\ref{subsection_disk_algebra}, it follows that ${\tt
    w}(f_d)=0$.
\end{proof}

\noindent  Besides $A(\mD^n)$ itself, another example of such an $R$
is $RH^\infty(\mD^n)$, the set of all rational functions without
  poles in $\overline{\mD^n}$.

In the definition of the $\nu$-metric given in
Definition~\ref{def_nu_metric} corresponding to
Lemma~\ref{lemma_polydisk_algebra}, the $\|\cdot\|_\infty$ now means
the following: if $F=(G,H)\in (C(\mT^n)\times C(\mT))^{p\times m}$,
then
$$
\|F\|_\infty =\max \bigg\{\displaystyle \max_{z\in \mT^n} \nm G(z)\nm,
\displaystyle\max_{w\in \mT} \nm H(w)\nm \bigg\} .
$$
This follows from \eqref{norm}, since the maximal ideal space
$\mathfrak{M}$ of the Banach algebra $S=C(\mT^n)\times C(\mT)$ can be identified with
$\mT^n \cup \mT$.

\begin{remark}
  By \cite{BruSas}, it follows that $A(\mD^n)$ is a projective free
  ring, since its maximal ideal space the polydisk $\overline{\mD^n}$
  is contractible.  Thus the set $\mS(A(\mD^n), p, m)$ of plants
  possessing a left and a right coprime factorization coincides with
  the class of plants that are stabilizable by
  \cite[Theorem~6.3]{Qua}.
\end{remark}

\begin{remark}

Roughly, the index function $\iota \colon \inv S \to G$ in all the examples given above (Sections
\ref{subsection_disk_algebra}--\ref{subsection_polydisk}) can be viewed as generalizations of the winding number for a
continuous nonvanishing function on the unit circle.  Another important application of such index functions, apart from
robust control theory as presented here, is to the Fredholm theory of various classes of operators (e.g., Toeplitz,
Wiener-Hopf, convolution)  associated with the function.  In this context we mention that Murphy \cite{Mur} has given
an abstract quantized $C^*$-algebra setting which, among other things,  unifies the connection between analytic index
and Fredholm index for the $C({\mathbb T})$-setting of Section \ref{subsection_disk_algebra} and the $AP$-setting of
Section~\ref{subsection_AP}.  There has also been a substantial amount of other work (see the books \cite{BS, BK})
where the analytic index has been extended to more general classes of functions (e.g. piecewise-continuous) in order to
develop the Fredholm theory for more general classes of Toeplitz operators.  On the other hand, the index theory for
semi-almost periodic symbols (a version of the Callier-Desoer class where $\widehat{f}_a$ is only required to be
continuous on the extended imaginary line and where $f -{\widehat f}_a$ is required only to be $AP$ rather than $APW$)
follows a different more complicated path rather than making use of the index function $W$ as in \eqref{Windex}.
Similarly, the Fredholm theory for Toeplitz operators on the quarter plane (associated with continuous functions on the
bitorus ${\mathbb T}^2$) (see \cite[Chapter~8]{BS}) makes use of the ${\mathbb Z}^2$-valued index associated with the
winding number of a function $f$ on ${\mathbb T}^2$ taken with respect to each variable separately, rather than with
the index $\iota$ as in Lemma~\ref{lemma_polydisk_algebra}.
\end{remark}

\section{Further directions}  \label{section_further}

\noindent It was shown in \cite{Vin} that when $R$ comprised rational functions
without poles in the closed unit disk, then the bound established in
Theorem~\ref{theorem_rst} is the best possible one in the following sense:
\medskip

\parbox[r]{12cm}{(P$'$): $C$ satisfying $\mu_{P_0,C}>m$ stabilizes $P$
  only if $d_{\nu}(P,P_0)\leq m$.}

\medskip

\noindent Since this property of $d_\nu$ already holds in the rational case, we
expect the same to hold also in the specific examples considered in
the previous section. We leave the question of investigation of
whether the property (P$'$) always holds in our abstract setup for
future work.


\begin{thebibliography}{10}


\bibitem{Ber}
D. Bernstein.
{\em Matrix Mathematics.
Theory, Facts, and Formulas with Application to Linear Systems Theory.}
Princeton University Press, Princeton, NJ, 2005.



\bibitem{BruSas}
A. Brudnyi and A.J. Sasane.
Sufficient conditions for the projective freeness of Banach algebras.
{\em Journal of Functional Analysis},  257:4003-4014, no.~12, 2009.

\bibitem{Bot}
A. B\"ottcher.
On the corona theorem for almost periodic functions.
{\em Integral Equations Operator Theory}, 33:253-272, no. 3, 1999.

\bibitem{BK}
A. B\"ottcher and Y.I. Karlovich. {\em Carleson curves, Muckenhoupt weights and Toeplitz operators},
Progress in Mathematics, Vol. 154, Birkh\"auser, Basel, 1997.

\bibitem{BKS}
A. B\"ottcher, Y.I. Karlovich, and I. Spitkovsky.
{\em Convolution Operators and Factorization of Almost Periodic Matrix Functions}.
Operator Theory Advances and Applications, Vol. 131, Birkh\"auser, Basel, 2002.

\bibitem{BS}
A. B\"ottcher and B. Silbermann.
{\em Analysis of Toeplitz Operators}, Springer, Berlin, 1990;
Second Edition (prepared  jointly with A. Karlovich), Springer, Berlin, 2006.

\bibitem{CalDes}
F.M. Callier and C.A. Desoer.
A graphical test for checking the stability of a linear time-invariant feedback system.
{\em IEEE Transactions on Automatic Control}, AC-17:773-780, no. 6, 1972.

\bibitem{CalDes0}
F.M. Callier and C.A. Desoer.
An algebra of transfer functions for distributed linear time-invariant systems.
{\em Special issue on the mathematical foundations of system theory.
IEEE Transactions on Circuits and Systems}, 25:651-662, no. 9, 1978.

\bibitem{Dav}
J.H. Davis.
Encirclement conditions for stability and instability of feedback systems with delays.
{\em International Journal of Control}, 15:793-799, no. 4, 1972.

\bibitem{DecMurSae}
R.A. DeCarlo, J. Murray and R. Saeks.
Multivariable Nyquist theory.
{\em International Journal of Control}, 25:657-675, no. 5, 1977.

\bibitem{Gam}
T.W. Gamelin.
{\em Uniform Algebras}.
Prentice-Hall, Englewood Cliffs, N.J., 1969.

\bibitem{GelRaiShi}
I. Gelfand, D. Raikov and G. Shilov.
{\em Commutative Normed Rings}.
Translated from the Russian, with a supplementary chapter.
 Chelsea Publishing Co., New York, 1964.

\bibitem{GohFel}
I.C. Gohberg and I.A. Fel'dman.
Integro-difference Wiener-Hopf equations. (Russian)
{\em Acta Sci. Math. (Szeged)},  30:199-224, 1969.

\bibitem{HilPhi}
E. Hille and R.S. Phillips.
{\em Functional Analysis and Semi-groups}.
Third printing of the revised edition of 1957.
American Mathematical Society Colloquium Publications,
Vol. XXXI. American Mathematical Society, Providence, R.I.,
1974.

\bibitem{JesTor}
B. Jessen and H. Tornehave.
Mean motions and zeros of almost periodic functions.
{\em Acta Mathematica}, 77:137-279, 1945.

\bibitem{Mur}
G.J. Murphy.
Topological and analytical indices in $C^*$-algebras.
{\em Journal of Functional Analysis}, 234:261-276, no. 2, 2006.

\bibitem{Qua}
A. Quadrat.
The fractional representation approach to synthesis problems:
an algebraic analysis viewpoint. II. Internal stabilization.
{\em SIAM Journal on Control and Optimization}, no. 1,
42:300-320, 2003.

\bibitem{Qua}
A. Quadrat.
A lattice approach to analysis and synthesis problems.
{\em Mathematics of Control, Signals, and Systems}, no. 2, 18:147-186, 2006.

\bibitem{Rud69}
W. Rudin.
{\em Function Theory in Polydiscs}.
W.A. Benjamin, New York-Amsterdam, 1969.

\bibitem{Rud91}
W. Rudin.
\newblock {\em Functional Analysis}.
\newblock 2nd Edition, McGraw Hill, 1991.

\bibitem{Sae}
R. Saeks.
On the encirclement condition and its generalization.
{\em IEEE Transactions on Circuits and Systems},
CAS-22:780-785, no. 10, 1975.

\bibitem{Sas}
A.J. Sasane.
An abstract Nyquist criterion containing old and new results.
{\em Submitted}.

\bibitem{Ull}
D.C. Ullrich.
{\em Complex Made Simple}.
Graduate Studies in Mathematics, 97, American Mathematical Society,
Providence, RI, 2008.

\bibitem{Vid}
M. Vidyasagar.
{\em Control System Synthesis: a Factorization Approach}.
MIT Press, 1985.

\bibitem{Vin}
G. Vinnicombe.
Frequency domain uncertainty and the graph topology.
{\em IEEE Transactions on Automatic Control},
no. 9, 38:1371-1383, 1993.

\bibitem{You}
N. Young.
Some function-theoretic issues in feedback stabilization.
{\em  Holomorphic spaces} (Berkeley, CA, 1995),  337-349,
Math. Sci. Res. Inst. Publ., 33, Cambridge Univ. Press, Cambridge, 1998.

\end{thebibliography}
\end{document}